\documentclass[11pt]{amsart}

\usepackage{bm}
\usepackage{fullpage}
\usepackage{amssymb}
\usepackage{amsmath,amsfonts,amsthm}
\usepackage{hyperref}
\usepackage{color}
\usepackage{cite}

\newtheorem{theorem}{Theorem}

\newtheorem{fact}{Fact}[theorem]
\newtheorem{definition}[theorem]{Definition}
\newtheorem{Claim}[theorem]{Claim}

\DeclareMathOperator\G{\mathcal{G}}

\DeclareMathOperator\C{\mathcal{C}}

\title{SOME BOUNDS  ON THE SIZE OF MAXIMUM $G$-FREE SETS IN GRAPHS}

\author{Yaser Rowshan$^1$}
\address{$^1$Y. Rowshan, 
	Department of Mathematics, Institute for Advanced Studies in Basic Sciences (IASBS), Zanjan 45137-66731, Iran}
\email{y.rowshan@iasbs.ac.ir}

\begin{document}
	\maketitle 
	
	\begin{abstract} 
	For given graph $H$, the independence number $\alpha(H)$ of  $H$, is the size of the maximum independent set of $V(H)$. Finding the maximum independent set in a graph is a NP-hard problem. Another version of the independence number is defined as the size of the maximum induced forest of $H$, and  called the forest number of $H$, and denoted  by $f(H)$. Finding $f(H)$ is also a NP-hard problem. Suppose that  $H=(V(H),E(H))$ be a graph, and $\G$
	be a family of  graphs, a graph $H$ has a $\G$-free $k$-coloring if there 
	exists a decomposition of  $V(H)$  into  sets $V_i$, $i-1,2,\ldots,k$, so that  $G\nsubseteq H[V_i]$ for each $i$, and  $G\in\G$.  $S\subseteq V(H)$ is $G$-free, where  the subgraph of $H$ induced by  $S$, be  $G$-free, i.e.  it contains no  copy of $G$.   Finding a maximum  subset of $H$, so that $H[S]$ be a $G$-free graph is a very hard problem as well. In this paper, we study the generalized version of the independence number of a graph. Also giving some  bounds   about the size of the maximum $G$-free subset of graphs is another purpose of this article.
	 
	\end{abstract}
	
	\section{Introduction} 
	All graphs considered here are undirected, simple, and finite graphs. For  given graph  $H=(V(H),E(H))$, its maximum degree and minimum degree are denoted by $\Delta(H)$ and $\delta(H)$, respectively. 
	 The degree and neighbors of $v$ in $H$, denoted by $\deg_H{(v)}$ ($\deg{(v)}$) and $N_H(v)$, respectively.  
	Suppose that $H$ be a graph, and let $V$ and $V'$ be two disjoint subsets of $V(H)$. Suppose that $W$ is any subset of $V(H)$, the induced subgraph $H[W]$ is the graph whose vertex set is $W$ and whose edge set consists of all of the edges in $E(H)$ that have both endpoints in $W$. The set $E(V, V' )$ is
	the set of all the edges $vv'$, which $v\in V$ and $v'\in V'$. Recall that an independent set  is a set of vertices in a graph, no two of which are adjacent.  A  maximum independent set in a graph is an independent set in which the graph contains no larger independent set. The independent number of a graph $H$ is the cardinality of a maximum independent set, and denoted by $\alpha(H)$. This problem was solved by Erd\"{o}s, and after that by Moon and Moser in \cite{moon1965cliques}. There are very few works about counting the number of maximum independent sets, see \cite{hopkins1985graphs, hujtera1993number, jou2000number, lowenstein2011independence} and \cite{jou2000number}. Finding a maximum independent set in a graph is a NP-hard problem
 
	Another version of the independence number is the forest number of a graph. Let $H$ be a graph, and $S\subseteq V(H)$, if $H[S]$ is acyclic, then $S$  is called   the induced forest of $H$. The forest number of a graph $H$ is the size of a maximum induced forest of $H$, and is denoted by $f(H)$. The decycling number $\phi(H)$ of a graph $H$  is the smallest number of vertices which can be removed from $H$ so that the resultant graph contains no cycle.. Thus, for a graph $H$ of order $n$, $\phi(H) + f(H) = n$. The decycling number was first proposed by Beineke and Vandell \cite{beineke1997decycling}. There is a fairly large literature of papers dealing with the forest number of a
	 graph. See for example \cite{alon2001large, bau2002decycling, punnim2011forest}, and \cite{zheng1990maximum}.
	 %%%%%%%%%%%%%%%%%%%%%%%%%%%
	 
	 The first item of the next results is attributed to P.K. Kwok and has come as an exercise in \cite{west2001introduction}, and the second item discussed in \cite{borg2010sharp}.
	 Suppose that $|V(H)|=n$, $\Delta(H)=\Delta$ and $|E(H)|=e$, then:
	   \begin{itemize}
	  	\item(Kwok Bound): $\alpha(H)\leq n-\frac{e}{\Delta},$
	  	\item(Borg Bound): $\alpha(H)\leq n-\lceil	\frac{n-1}{\Delta} \rceil.$
	  \end{itemize}
	 
	 %%%%%%%%%%%%%%%%%%%%%%%%%%%%%%%%%%%%%%%%%%%%%%
	 The Borg Bound is an efficiently countable  upper bound for the $\alpha(H)$.  However, generally gives an  estimation, greater than or equal to the Kwok Bound.

	 %%%%%%%%%%%%%%%%%%%%%%%%%%%%%%%%%%%%%%%%%%%%%%%%%%%%%%%%%%%%	
	 
	 %========================================== 
	 %%%%%%%%%%%%%%%%%%%%%%%%%%%%%%%%%%%%%%
	 \begin{theorem}\label{t1}\cite{henning2014new} Let $H$   is a graph with $n$ member, and $p$ is an integer, so that $(A)$ holds, then:
	 	thus:
	 	\[ \alpha(H) \geq \frac{2n}{p}.\]
	 	
	 	$(A)$: For any clique $K$ in $H$ there exists a member of  $V(K)$ say $v$, so that $\deg(v)\leq p-|V(K)|-1$.
	 \end{theorem} 
	 %%%%%%%%%%%%%%%%%%%%%%%%%%%%%%%%%%%%%%

	\subsection{$G$-free coloring.}	\hfilneg
	
	The conditional chromatic number $\chi(H,P)$ of $H$,  is the smallest  integer  $k$, for which there 
	exists a decomposition of  $V(H)$  into  sets $V_i$, $i-1,2,\ldots,k$, so that for each $i$, $H[V_i]$ satisfies the property $P$, where   $P$ is a graphical property and $H[V_i]$ is the induced subgraph on $V_i$.
	Harary in 1985 presented this extension of graph coloring ~\cite{MR778402}. Suppose that $\G$
	be a family of  graphs, when $P$ is the feature that a subgraph induced by each color class does not contain  any copy of 
	members of $\G$, we write $\chi_{\G}(H)$ instead of $\chi(H, P)$. A graph $H$ has a $\G$-free $k$-coloring if there 
	exists a decomposition of  $V(H)$  into  sets $V_i$, $i-1,2,\ldots,k$, so that for each $i$, $H[V_i]$ does not include  any copy of 
    the members of $\G$. For simplicity of notation, if $\G=\{G\}$, then we write $\chi_G(H)$ instead of $\chi_{\G}(H)$. An ordinary  $k$-coloring of $H$  can be viewed as $\G$-free $k$-coloring of a graph $H$ by taking $\G=\{K_2\}$.   
	
	%============================================
 For any two graphs $H$ and $G$, recall that $\chi_G(H)$ is the $G$-free chromatic number of the graph $H$, now suppose that  $S$ is a maximum subset  of $V(H)$, so that $H[S]$ is $G$-free, therefore it is easy to say that $|S|\geq \frac{n(H)}{\chi_G(H)}$. By considering $H=K_6$ and $G=K_3$, one can check that $\chi_G(H)=3$, that is $|S|=2= \frac{n(H)}{\chi_G(H)}$, which means that this bond is sharp. Set $\G=\{C_n, n\geq 3\}$ and let $H$ be a graph, therefore  one can say that $|S|=f(H)$, where $H[S]$ is $\G$-free and $S$ has the maximum size  possible.  In this article, we prove  results as follow:
 \begin{theorem}\label{t2}	Let $H$ and $G$ are two graphs, where $|V(H)|=n, \Delta(H)=\Delta, |E(H)|=e_H, |E(G)|=e_G$  and $\delta(G)=\delta$. Then:
 \begin{itemize}
 	
 	\item   $|S|\geq n+\frac{e_G+e_{H'}-e_H-\Delta}{\delta}.$	
 
 	\item $|S|\geq  n-\frac{\delta n_{\delta}(S)+(\delta+1)n_{\delta+1}(S)+\ldots+\Delta n_{\Delta}(S)}{\delta}.$
 	\item $|S|\leq  n-\frac{\delta n_{\delta}(S)+(\delta+1)n_{\delta+1}(S)+\ldots+\Delta n_{\Delta}(S)}{\Delta}.$
 	 
\end{itemize}
\end{theorem}
 
	%%%%%%%%%%%%%%%%%%%%%%%%%%%%%%%%%%%%%%%%%%%%%%%%%%%%%%%%%%%
		\begin{theorem}[Main theorem]\label{t4}
		Let $H$ and $G$ are two graphs, where $|V(H)|=n$ and $\delta(G)=\delta$. Suppose that $P$ is a positive integer,  which for each connected component  $X\in R(H)$, there exists a vertex of $X$ say $x$, so that $\deg_H(x)\leq P-|X|-\delta$. Then:
		\[|S|\geq \frac{(\delta+1)n}{P}\]
	 Where $S$ has  the maximum size  possible and $H[S]$ is $G$-free.
	\end{theorem}
	
	\section{Main results}
Let $H$ and $G$ are two graphs, in this section, we  give some  upper and lower bounds on  the size of the maximum $G$-free subset of $H$. Next results   offers further investigations to get some good bounds on the size of the  maximum $G$-free subset of $H$ and exact results, when  feasible. The next results are examples of two graphs $H$ and $G$, in which the maximum $G$-free subsets of $H$ easily obtained.
 \begin{itemize}
 	\item If $|V(H)|=n$, then $|S|$ = $n$ if and only if (iff) $H$ is $G$-free. 
 	\item If $|V(H)|=n$, then $|S|$ = $n-1$ iff either $H\cong G$ or $|V(G)|=n$ and $G\subseteq H$. 
 	\item If $H\cong K_n$, and $G$ has $m$ members, where $m\leq n$, then $|S|= m-1$.  
 	\item If $|V(H)|=n$ and $G$ has $n-1$ members, then $|S|= n-2$  iff $G\subseteq H\setminus\{v\}$ for each $v\in V(H)$.  
 \end{itemize}
  In the next two theorems, we  give  lower and upper bounds  on  the size of the maximum $G$-free subset of $H$, in terms of the number of
vertices and edges, maximum degree, and minimum degree of $H$ and $G$. 
	\begin{theorem}\label{th1}
		Suppose that $H$ and $G$ are two graphs, where $|V(H)|=n, \Delta(H)=\Delta, |E(H)|=e_H, |E(G)|=e_G$,  and $\delta(G)=\delta$. Suppose that $S \subseteq V(H)$ is maximum $G$-free. Then:
		\[|S|\geq n+\frac{e_G+e_{H'}-e_H-\Delta}{\delta}.\]
	Where $e_{H'}=| E(H[V(H)\setminus S])|$.  
	\end{theorem}
	\begin{proof}
Suppose that  $S\subseteq V(H)$, where $H[S]$ is $G$-free and $S$ has  maximum size as possible, and $|S|=m$. As $S$ is maximum, for each $v\in V(H)\setminus S$
 so $H[S\cup\{v\}]$ contains at least one copy of $G$. Therefore, since $\deg_H(v)\leq \Delta$, thus $E(H[S])\geq e_G-\Delta$. As each vertex of $V(H)\setminus S$ has at least $\delta$ neighbors in $S$, so we have:
\[e_H\geq (n-m)\delta +e_G-\Delta +e_{H'} .\]
Thus, it can be checked that:
\[ m\delta\geq n\delta +e_G+e_{H'}-e_H-\Delta.\]
Hence, $m\geq n+\frac{e_G+e_{H'}-e_H-\Delta}{\delta}$. Which means that the proof is complete.		
	\end{proof}	
\begin{theorem}\label{th2}
	Suppose that $H$ and $G$ are two graphs, where $H$ has $n$ members, $\Delta(H)=\Delta$. Let  $S$ be a  maximum $G$-free subset of $V(H)$. Then:
	\[  n-\frac{\delta n_{\delta}(S)+(\delta+1)n_{\delta+1}(S)+\ldots+\Delta n_{\Delta}(S)}{\delta}\leq |S|\leq    n-\frac{\delta n_{\delta}(S)+(\delta+1)n_{\delta+1}(S)+\ldots+\Delta n_{\Delta}(S)}{\Delta}.\]
	
\end{theorem}
\begin{proof}
Assume that  $S\subseteq V(H)$, where $H[S]$ is $G$-free and $S$ has  maximum size as possible, and $|S|=m$. Suppose that $n_i(S)$ be the vertices of $V(H)\setminus S$, which has exactly $i$ neighbors in $S$.  Now by maximality of $S$,  it is easy to check that $n_i(S)=0$ for $i=0,1,2,\ldots, \delta-1$. So:
\begin{equation}\label{e0}
	n-m=n_{\delta}(S)+n_{\delta+1}(S)+\ldots+n_{m}(S).
\end{equation}
As each vertex of $V(H)\setminus S$ has at least $\delta$ and at most $\Delta$ neighbors in $S$, then:
\begin{equation}\label{e00}
	\delta (n-m)\leq \delta n_{\delta}(S)+(\delta+1)n_{\delta+1}(S)+\ldots+\Delta n_{\Delta}(S)=\sum_{v\in V(H)\setminus S}\deg(v)\leq (n-m)\Delta.
\end{equation}
	 Therefore  by Equation \ref{e00}, it can be checked that:
	 \begin{itemize}
	 	\item $m\geq  n-\frac{\delta n_{\delta}(S)+(\delta+1)n_{\delta+1}(S)+\ldots+\Delta n_{\Delta}(S)}{\delta}.$
	 	\item $m\leq  n-\frac{\delta n_{\delta}(S)+(\delta+1)n_{\delta+1}(S)+\ldots+\Delta n_{\Delta}(S)}{\Delta}.$
	 \end{itemize}

	   Which means that the proof is complete.		
\end{proof}	

Combining Theorems \ref{th1} and \ref{th2}, The correctness of Theorem \ref{t2} is obtained . In the next theorem, we generalize Theorem \ref{t1} to specify an appropriate  lower bound for the size of the maximum $G$-free subgraphs of graph H. We need to determine a series of special subgraphs of $H$, which are expressed in the following definition:

		\begin{definition}\label{d1}
		Let $H$ and $G$ be two graphs, where $|V(H)|=n$ and $\delta(G)=\delta$.  Suppose that $S$ be the  maximum subset of $V(H)$ so that $H[S]$ is $G$-free. Therefore as $S$ has maximum size and $H[S]$ is $G$-free, then for each $v\in V(H)\setminus S$, it can be say that $H[S\cup\{v\}]$, contain at least one copy of $G$. In other word for each $v\in V(H)\setminus S$, there is at least one copy of $G-v$ in $H[S]$. Now, for each $v\in V(H)\setminus S$ define $A_v$ as follows:
		\[A_v=\{G^i_v, ~~G^i_v\cong G-v\subseteq H[S]\}.\]
		Assume that $N^i_v=N(v)\cap G^i_v$, for each $i\in\{1,\ldots, |A_v|\}$.  Now for each  $i\in\{1,\ldots, |A_v|\}$, define $M^i_v$ as follow:
		\[M^i_v=\{u\in V(H)\setminus S, ~N(u)\cap S=N_v^i\}.\]
		
		Therefore, define $R(H)$ as follow:
		\[R(H)=\{ Clique ~of ~H[M^i_v],~ for ~each~~ i\in[A_v], ~ and~~each~~v\in V(H)\setminus S\}.\]
		Where $[A_v]=\{1,2,\ldots,|A_v|\}$.
	\end{definition}
 To prove the next results, we present an argument that is similar to the proof of Theorem \ref{t1} in \cite{henning2014new}. However, in our proof, we carefully choose a maximum $G$-free set $S$ in the graph $H$, so that $|E(S, V(H)\setminus S)|$   is minimize and $H[S]$ is $G$-free. With this choice of $S$, we establish a property on  $H$ by considering the operation of replacing a vertex in $S$ with  $V(H)\setminus S$, to get a smaller number of edges between $V(H)\setminus S$ and $S$.
	\begin{theorem}\label{th3}
		Let $H$ and $G$ are two graphs, where $|V(H)|=n$ and $\delta(G)=\delta$. Suppose that $P$ is a positive integer,  where for each  $X\in R(H)$, there exists a vertex of $X$ say $x$, so that $\deg_H(x)\leq P-|X|-\delta$. Then:
		\[|S|\geq \frac{(\delta+1)n}{P}\]
		 Where $S\subseteq V(H)$, and $S$ has  the maximum size  possible, so that $H[S]$ is $G$-free.
	\end{theorem}
	\begin{proof}
		Suppose that  $m$ is the size of the  maximum $G$-free subset of $V(H)$. Now set $A$ as follow: 
		\[A=\{S\subseteq V(H), G\nsubseteq H[S], |S|=m\}.\]
		Therefore, for any member of $A$, say $S$, we define $\beta(S)$ as fallow:
			\begin{equation}\label{e}
		\beta(S)=\sum_{y\in V(H)\setminus S} |N(y)\cap S|=\sum_{x\in S}|N(x)\cap (V(H)\setminus S)|.
		\end{equation}
		
		In other word, $\beta(S)=|E(S,V(H)\setminus S)|$.  Now we define $B$ as follow: 
		\[B=\{\beta(S), S\in A\}.\]
		Let $\beta$ is a minimal members of $B$, and without loss of generality  suppose that $\beta=\beta(S^*)$, that is $|E(S^*,V(H)\setminus S^*)|$ is minimize. Assume that $\gamma_i(S^*)$ be the vertices of $V(H)\setminus S^*$, such that its vertices have exactly $i$ neighbors in $S^*$. Since $S^*\in A$, it is easy to check that $\gamma_i(S^*)=0$ for $i=0,1,2,\ldots, \delta-1$. Hence, one can say that:
		\begin{equation}\label{e1}
			|V(H)\setminus S|=n-m=\gamma_{\delta}(S^*)+\gamma_{\delta+1}(S^*)+\ldots+\gamma_{m}(S^*).
		\end{equation}
		
		Furthermore, by considering   $\beta(S^*)=|E(S^*,V(H))\setminus S^*|$, and by Equation \ref{e} and \ref{e1} it is easy to say that:
		\begin{equation}\label{e2}
			\sum_{y\in S^*}|N(y)\cap (V(H)\setminus S^*)|=\delta\gamma_{\delta}(S^*)+(\delta+1)\gamma_{\delta+1}(S^*)+\ldots+m\gamma_{m}(S^*)=\sum_{i=\delta}^m i\gamma_i(S^*).
		\end{equation}
		
		Multiplying Equation \ref{e1} by $\delta+1$, and subtracting Equation \ref{e2}, we acquire  the next:
		\[(\delta+1)(n-m)-\sum_{y\in S^*}|N(y)\cap (V(H)\setminus S^*)|\]
		\[=(\delta+1)(\gamma_{\delta}(S^*)+\gamma_{\delta+1}(S^*)+\ldots+\gamma_{m}(S^*))-\sum_{y\in S^*}|N(y)\cap (V(H)\setminus S^*)| \]
		\begin{equation}\label{e3}
			=\gamma_{\delta}(S^*)-\gamma_{\delta+2}(S^*)-\ldots-(m-(\delta+1))\gamma_{m}(S^*)\leq \gamma_{\delta}(S^*).
		\end{equation}
		
		As $S^*\in A$, therefore by maximality of $S^*$, for each vertex of  $ V(H)\setminus S$ say $v$, one can say that $H[S\cup\{v\}]$ contains at least one copy of $G$, namely $G_v$. Suppose that $v'$ be a  vertex of $V(G_v)$ with minimum degree in $G_v$. Let $N(v')\cap V(G')=X_{\delta}$.  Hence $X_{\delta}$ is a fixed subset of $S^*$ where $|X_{\delta}|=\delta$. Now we define $Y_{X_{\delta}}$ as fallow:
		\begin{equation}\label{e4}
			Y_{X_{\delta}}=\{w\in V(H)\setminus S^*, N(w)\cap S^*=X_{\delta}\}.
		\end{equation}
		In other word, assume that $	Y_{X_{\delta}}$  is the set of all vertices in $V(H)\setminus S^*$, so that adjacent to each vertex of $X_{\delta}$ but  no other vertices of $S^*\setminus X_{\delta}$, so every vertex in $Y_{X_{\delta}}$ has no neighbor in $S^*\setminus X_{\delta}$.
		Therefore, we have the next claim.
		\begin{Claim}\label{cl1}
			 $H[Y_{X_{\delta}}]$ is a clique in $H$.
		\end{Claim}
	 \begin{proof}
	 	 By contradiction, suppose that there exist at least two vertices of $Y_{X_{\delta}}$, say $x,x'$ so that $xx'\notin E(H)$. Therefore, since $S^*\in A$, one can check that $H[S^*\cup\{x\}]$ and $H[S^*\cup\{x'\}]$ contains at least one copy of $G$, say $G_x$ and $G_{x'}$, respectively. Now, suppose that $x''\in X_{\delta}$, and set $S'=S^*\setminus\{x''\}\cup \{x,x'\}$, hence it is easy to see that:
	 	 \[|S'|=|S^*|+1.\]
	 	 And $|N(x)\cap S'|=|N(x')\cap S'|=\delta-1$, that is  $H[S']$ is $G$-free, a contradiction to maximality of $S^*$. So $xx'\in E(H)$ for each $x,x'\in Y_{X_{\delta}}$, that is $H[Y_{X_{\delta}}]$ is a clique in $H$.
	 \end{proof}
	 Therefore, by Claim \ref{cl1} it can be  checked that $Y_{X_{\delta}}\in R(H)$. Now for a fixed vertex of  $ X_{\delta}$ say $x$, assume that: 
		\begin{equation}\label{e5}
			|N(x)\cap(V(H)\setminus S^*)| +|Y_{X_{\delta}}|+\delta \geq P.
		\end{equation}
		
		So, by considering the vertices of $ Y_{X_{\delta}}$, we have the following claim:
			\begin{Claim}\label{cl2}
		For each $w\in Y_{X_{\delta}}$, we have	$|N(w)\cap (V(H)\setminus S^*)|\geq |N_H(x)|-|N(x)\cap  S^*|$.
		\end{Claim}
		\begin{proof}
		By contradiction, suppose that there exists a vertex of $Y_{X_{\delta}}$ say $w$ so that:
		\[|N(w)\cap (V(H)\setminus S^*)|\leq |N_H(x)|-|N(x)\cap  S^*|-1.\]
		Hence, as $w\in Y_{X_{\delta}}$, so $|N(w)\cap S^*|=\delta$. Therefore, it can be  checked that $|N_H(w)|\leq |N_H(x)|-|N(x)\cap  S^*|-\delta-1$.
		Then set $S''=S^*\setminus \{x\}\cup \{w\}$, hence seeing $|S''|=|S^*|$ is obvious. Also as $w\in Y_{X_{\delta}}$, $wx\in E(H)$, and $x\in S^*$, it can be  checked that $H[S'']$ is $G$-free.  Now by considering $\beta(S'')$ we have the following fact:
		\begin{fact}\label{f1}
		 $\beta(S'')\leq \beta(S^*)-1$.
		\end{fact} 
	{\bf proof of the fact:} As $S''=S^*\setminus \{x\}\cup \{w\}$, one can say that $\beta(S'')= \beta(S^*) +|N(x)\cap  S^*|-|N(x)\cap (V(H)\setminus S^*)|+|N(w)\cap (V(H)\setminus S^*)|-|N(w)\cap  S^*|$. As $|N(w)\cap  S^*|=\delta$, and $|N(w)\cap (V(H)\setminus S^*)|\leq |N_H(x)|-|N(x)\cap  S^*|-1$, one can check that $|N(x)\cap  S^*|-|N(x)\cap (V(H)\setminus S^*)|+|N(w)\cap (V(H)\setminus S^*)|-|N(w)\cap  S^*|\leq-\delta$, that is: 
		\begin{equation}\label{e6}
	\sum_{x\in S''}|N(x)\cap (V(H)\setminus S'')|=\beta(S'')\leq \sum_{x\in S^*}|N(x)\cap (V(H)\setminus S^*)|-\delta=\beta(S^*)-\delta.
		\end{equation}
		
	Therefore by Fact \ref{f1},  $\beta(S'')\leq \beta(S^*)-1$, where $S''\in A$.	A contradiction  to minimality of $\beta(S^*)$. Hence 	for each $w\in Y_{X_{\delta}}$, we have	$|N(w)\cap (V(H)\setminus S^*)|\geq |N_H(x)|-|N(x)\cap  S^*|$. Which means that the proof of the claim is complete.
	\end{proof}	
		 Therefore, by Claim \ref{cl2}, for each $w\in  Y_{X_{\delta}}$ and each $x\in X_{\delta}$, we have the following equation:
		\[|N(w)\cap (V(H)\setminus S^*)|\geq |N_H(x)|-|N(x)\cap  S^*|=|N(x)\cap (V(H)\setminus S^*)|.\]
		Therefore:
		\begin{equation}\label{e7}
			|N_H(w)|=|N(w)\cap (V(H)\setminus S^*)|+|N(w)\cap  S^*|\geq |N(x)\cap (V(H)\setminus S^*)|+1.
		\end{equation}
		Since by Equation \ref{e5},	$|N(x)\cap(V(H)\setminus S^*)|\geq P -|Y_{X_{\delta}}|-\delta$,  by Equation \ref{e7} it can be  checked that:
		\begin{equation}\label{e8}
			|N_H(w)| \geq  P -|Y_{X_{\delta}}| -\delta+1.
		\end{equation}
	  
		Therefore, as $Y_{X_{\delta}}\in R(H)$ and by Equation \ref{e8}, $\deg(y)\geq  P +1-|Y_{X_{\delta}}|-\delta$, for each $y\in Y_{X_{\delta}}$, which is a contradiction to assumption. Hence $|N(x)\cap(V(H)\setminus S^*)| \leq P-|Y_{X_{\delta}}|-\delta-1$ for each $x\in X_{\delta}$. As $X_{\delta}\subseteq S^*$,  so by Equation \ref{e3},
		\[(\delta+1)n\leq \gamma_{\delta}(S^*)+\sum_{y\in S^*}|N(y)\cap (V(H)\setminus S^*)|+(\delta+1)m\]
		\[	\leq\sum_{y\in S^*}(|Y_{X_{\delta}}|+|N(y)\cap(V(H)\setminus S^*)|+\delta+1)\leq |S^*|.P=mp.\] 
		Therefore, $(\delta+1)n\leq m.P$, thus:
		\[|S^*|=m\geq \frac{(\delta+1)n}{P}\]
		Which means that the proof is complete.
		
	\end{proof}
 In Theorem \ref{th3}, if we take $G=K_2$, then we get Theorem \ref{t1}. By setting $\G=\{C_n, n\geq 3\}$ and any arbitrary graph for $H$, it is easy to say that $|S|=f(H)$, where $S$ is the  maximum subset of $V(H)$, so that $H[S]$ is $\G$-free and $f(H)$ is the forest number of $H$. In particular, by setting $\C$ as $2$-regular connected graph in Theorem \ref{th3}, we can show that the following result is true:
	\begin{theorem}
	Let $H$ and $G$ are two graphs, where $|V(H)|=n$ and $G\in \C$, that is $G$ is a 2-regular connected  graph(cycle). Suppose that $P$ is a positive integer,  where for each $X\in R(H)$ there exists a vertex of $X$ say $x$, such that $\deg_H(x)\leq P-|X|-2$. Then:
	\[f(H)\geq \frac{3n}{P}\]
	Where $f(H)$ is  the forest number of $H$ and $R(H)$ defined in \ref{d1}.
\end{theorem}
	%---------------------------------------------------------------------------------------%
Suppose that $\G_k=\{G,~~ \omega(G)=k\}$, and $H$ be any such graph.  Also, assume that $I$ is the maximum independent set in $H$. In the following theorem,  a suitable lower bound for the size of  the maximum $\G_k$-free subgraph of the graph $H$ is determined.	
	
	%=========================================================================
\begin{theorem}
	
Suppose that $H=H_1$ is a graph,  and $I_1$ is the maximum independent set in $H_1$ where $|I_i|=i_1$. For each $2\leq j$, set $H_j=H_{j-1}\setminus I_{j-1}$, and set $I_j$ as the maximum independent set in $H_j$, where $|I_j|=i_j$ for each $1\leq j$. Then:
	\[|S|\geq \sum_{j=1}^{j=k-1} i_j \]
 Where $S$ has the maximum size possible, and $H[S]$ is $\G_k$-free.
\end{theorem}	
\begin{proof}
	Since, for each $j$,  $I_j$ is the maximum independent set in $H_j$, then for each $n$, one can check that $|\omega (H[\cup_{j=1}^{j=n} I_j])|\leq n$. Therefore, for $n=k-1$, we have $|\omega (H[\cup_{j=1}^{j=k-1} I_j])|\leq k-1$. Now,  as $\G_k=\{G,~~ \omega(G)=k\}$, so $U$ is a $G$-free subset of $H$ for each $G\in \G_k$, where $U=\cup_{j=1}^{j=k-1} I_j$. Hence, it is clear to see that $U$ is a $\G_k$-free subset of $H$. As $|U|=\sum_{j=1}^{j=k-1} i_j$, so:
		\[|S|\geq |U= \sum_{j=1}^{j=k-1} i_j \]
		Which means that the proof is complete.
\end{proof}
 
	%========================================================================================
	%%%%%%%%%%%%%%%%%%%%%%%%%%%%%%%%%%%%%%%%%
	\bibliographystyle{plain}
	\bibliography{G-free3}
\end{document}